\documentclass[a4paper,twoside,11pt]{article}
\usepackage{a4wide, fancyhdr, amsmath, amssymb, mathtools, yfonts}
\usepackage{graphicx}
\usepackage{tikz}
\usepackage[all]{xy}
\usepackage[utf8]{inputenc}
\usepackage{amsthm}
\usepackage[english]{babel}
\usepackage{chngcntr}
\usepackage{ifthen}
\usepackage{calc}
\usepackage{hyperref}
\usepackage{authblk}
\numberwithin{equation}{section}

\newcommand{\Z}{\mathbb{Z}}
\newcommand{\Q}{\mathbb{Q}}

\newcommand{\OO}{\mathcal{O}}

\newcommand\FF{\mathbb{F}}
\newcommand\MM{\mathbb{M}}
\newcommand\DD{\mathcal{D}}

\newcommand\aaa{\mathfrak{a}}

\newcommand\AAA{\mathfrak{A}}
\newcommand\bb{\mathfrak{b}}
\newcommand\BBB{\mathfrak{B}}

\newcommand\ccc{\mathfrak{c}}

\newcommand\mm{\mathfrak{m}}

\newcommand\ff{\mathfrak{f}}
\newcommand\nn{\mathfrak{n}}

\newcommand\pp{\mathfrak{p}}

\newcommand\qq{\mathfrak{q}}

\newcommand\bbb{\mathfrak{b}}

\newcommand\rk{\mathrm{rk}}
\newcommand\Vol{\mathrm{Vol}}

\newcommand\Gal{\mathrm{Gal}}
\newcommand\Norm{N}
\newcommand\spin{\mathrm{spin}}

\newcommand\ONE{\textbf{1}}
\newcommand\Cl{\mathcal{C}\ell}

\newtheorem{theorem}{Theorem}
\newtheorem{lemma}{Lemma}[section]

\newtheorem{corollary}[lemma]{Corollary}

\newtheorem*{conjecture}{Conjecture $C_n$}

\title{\vspace{-\baselineskip}\sffamily\bfseries Joint distribution of spins}
\author[1]{Peter Koymans\thanks{Niels Bohrweg 1, 2333 CA Leiden, Netherlands, p.h.koymans@math.leidenuniv.nl}}
\author[2]{Djordjo Milovic\thanks{Gower Street, WC1E 6BT London, United Kingdom, djordjo.milovic@ucl.ac.uk}}
\affil[1]{Mathematisch Instituut, Leiden University}
\affil[2]{Department of Mathematics, University College London}

\date{\today}

\begin{document}
\maketitle

\begin{abstract}
We answer a question of Iwaniec, Friedlander, Mazur and Rubin \cite{FIMR} on the joint distribution of spin symbols. As an application we give a negative answer to a conjecture of Cohn and Lagarias on the existence of governing fields for the $16$-rank of class groups under the assumption of a short character sum conjecture.
\end{abstract}

\section{Introduction}
One of the most fundamental and most prevalent objects in number theory are extensions of number fields; they arise naturally as fields of definitions of solutions to polynomial equations. Many interesting phenomena are encoded in the splitting of prime ideals in extensions. For instance, if $p$ and $q$ are distinct prime numbers congruent to $1$ modulo $4$, the statement that $p$ splits in $\Q(\sqrt{q})/\Q$ if and only if $q$ splits in $\Q(\sqrt{p})/\Q$ is nothing other than the law of quadratic reciprocity, a common ancestor to much of modern number theory.

Let $K$ be a number field, $\pp$ a prime ideal in its ring of integers $\OO_K$, and $\alpha$ an element of the algebraic closure $\overline{K}$. Suppose we were to ask, as we vary $\pp$, how often $\pp$ splits completely in the extension $K(\alpha)/K$. If $\alpha$ is fixed as $\pp$ varies over all prime ideals in $\OO_K$, a satisfactory answer is provided by the Chebotarev Density Theorem, which is grounded in the theory of $L$-functions and their zero-free regions. The Chebotarev Density Theorem, however, often cannot provide an answer if $\alpha$ varies along with $\pp$ in some prescribed manner. The purpose of this paper is to fill this gap for quadratic extensions in a natural setting that arises in many applications. This setting, which we now describe, is inspired by the work of Friedlander, Iwaniec, Mazur, and Rubin \cite{FIMR} and is amenable to sieve theory involving sums of type I and type II, as opposed to the theory of $L$-functions.

Let $K/\Q$ be a Galois extension of degree $n$. Unlike in \cite{FIMR}, we do \textit{not} impose the very restrictive condition that $\Gal(K/\Q)$ is cyclic. For the moment, let us restrict to the setting where $K$ is totally real and where every totally positive unit in $\OO_K$ is a square, as in \cite{FIMR}. To each non-trivial automorphism $\sigma\in\Gal(K/\Q)$ and each odd principal prime ideal $\pp\subset\OO_K$, we attach the quantity $\spin(\sigma, \pp)\in\{-1, 0, 1\}$, defined as
\begin{equation}
\label{eDSpin}
\spin(\sigma, \pp) = \left(\frac{\pi}{\sigma(\pi)}\right)_{K, 2},
\end{equation}
where $\pi$ is any totally positive generator of $\pp$ and $\left(\frac{\cdot}{\cdot}\right)_{K, 2}$ denotes the quadratic residue symbol in $K$. If we let $\alpha^2 = \sigma^{-1}(\pi)$, then $\spin(\sigma, \pp)$ governs the splitting of $\pp$ in $K(\alpha)$, i.e., $\spin(\sigma, \pp) = 1$ (resp., $-1$, $0$) if $\pp$ is split (resp., inert, ramified) in $K(\alpha)/K$. In \cite{FIMR}, under the assumptions that $\sigma$ generates $\Gal(K/\Q)$, that $n\geq 3$, and that the technical Conjecture $C_n$ (see Section~\ref{conjCn}) holds true, Friedlander et al.\ prove that the natural density of $\pp$ that are split (resp., inert) in $K(\sqrt{\alpha})/K$ is $\frac{1}{2}$ (resp., $\frac{1}{2}$), just as would be the case were $\alpha$ not to vary with $\pp$.

More generally, suppose $S$ is a subset of $\Gal(K/\Q)$ and consider the \textit{joint spin}
$$
s_{\pp} = \prod_{\sigma\in S}\spin(\sigma, \pp),
$$
defined for principal prime ideals $\pp = \pi\OO_K$. If we let $\alpha^2 = \prod_{\sigma\in S}\sigma^{-1}(\pi)$, then $s_{\pp}$ is equal to~$1$ (resp., $-1$, $0$) if $\pp$ is split (resp., inert, ramified) in $K(\alpha)/K$. If $\sigma^{-1}\in S$ for some $\sigma\in S$, then the factor $\spin(\sigma, \pp)\spin(\sigma^{-1}, \pp)$ falls under the purview of the usual Chebotarev Density Theorem as suggested in \cite[p.\ 744]{FIMR} and studied precisely by McMeekin \cite{McMeekin}. We therefore focus on the case that $\sigma\not\in S$ whenever $\sigma^{-1}\in S$ and prove the following equidistribution theorem concerning the joint spin $s_{\pp}$, defined in full generality, also for totally complex fields, in Section~\ref{sDefSpin}.

\begin{theorem}
\label{tJoint}
Let $K/\Q$ be a Galois extension of degree $n$. If $K$ is totally real, we further assume that every totally positive unit in $\OO_K$ is a square. Suppose that $S$ is a non-empty subset of $\Gal(K/\Q)$ such that $\sigma \in S$ implies $\sigma^{-1} \not \in S$. Foe each non-zero ideal $\aaa$ in $\OO_K$, define $s_{\aaa}$ as in \eqref{defSpin}. Assume Conjecture $C_{|S|n}$ holds true with $\delta = \delta(|S|n) > 0$ (see Section~\ref{conjCn}). Let $\epsilon > 0$ be a real number.  Then for all $X\geq 2$, we have
$$
\sum_{\substack{\Norm(\pp)\leq X \\ \pp\text{ prime}}} s_\pp \ll X^{1 - \frac{\delta}{54|S|^2n(12n+1)} + \epsilon},
$$ 
where the implied constant depends only on $\epsilon$ and $K$.
\end{theorem}
It may be possible to weaken our condition on $S$ and instead require only that there exists $\sigma \in S$ with $\sigma^{-1} \not \in S$. 

The main theorem in \cite{FIMR} is the special case of Theorem~\ref{tJoint} where $\Gal(K/\Q) = \langle\sigma\rangle$, $n\geq 3$, and $S = \{\sigma\}$. After establishing their equidistribution result, Friedlander et al.\ \cite[p.\ 744]{FIMR} raise the question of the joint distribution of spins, and in particular the case of $\spin(\sigma, \pp)$ and $\spin(\sigma^2, \pp)$ where again $\Gal(K/\Q) = \langle\sigma\rangle$, but $S = \{\sigma, \sigma^2\}$ and $n\geq 5$. The following corollary of Theorem~\ref{tJoint} applied to the set $S = \{\sigma, \sigma^2\}$ answers their question.

\begin{theorem}
\label{cDensities}
Let $K/\Q$ be a totally real Galois extension of degree $n$ such that every totally positive unit in $\OO_K$ is a square. Suppose that $S = \{\sigma_1, \ldots, \sigma_{t}\}$ is a non-empty subset of $\Gal(K/\Q)$ such that $\sigma \in S$ implies $\sigma^{-1} \not \in S$. Assume Conjecture $C_{tn}$ holds true (see Section~\ref{conjCn}). Let $\mathbb{e} = (e_1, \ldots, e_{t})\in\FF_2^{t}$. Then, as $X\rightarrow \infty$, we have
$$
\frac{|\{\pp\text{ principal prime ideal in }\OO_K:\ \Norm(\pp)\leq X,\ \spin(\sigma_i, \pp) = (-1)^{e_i}\text{ for }1\leq i\leq t\}|}{|\{\pp\text{ principal prime ideal in }\OO_K:\ \Norm(\pp)\leq X\}|}\sim \frac{1}{2^t}.
$$
\end{theorem}

We expect that Theorem~\ref{tJoint} has several algebraic applications; see for example the original work of Friedlander et al. \cite{FIMR}, but also \cite{KM1}, \cite{KM2}, and~\cite{Milovic}. Here we give one such application by giving a negative answer to a conjecture of Cohn and Lagarias~\cite{CL}. Given an integer $k\geq 1$ and a finite abelian group $A$, we define the $2^k$-rank of $A$ as
\[
\text{rk}_{2^k} A = \dim_{\FF_2} 2^{k - 1} A/2^k A.
\]
Cohn and Lagarias~\cite{CL} considered the one-prime-parameter families of quadratic number fields $\{\Q(\sqrt{dp})\}_p$, where $d$ is a fixed integer $\not\equiv 2\bmod 4$ and $p$ varies over primes such that $dp$ is a fundamental discriminant. Bolstered by ample numerical evidence as well as theoretical examples \cite{CL2}, they conjectured that for every $k\geq 1$ and $d\not\equiv 2 \bmod 4$, there exists a governing field $M_{d, k}$ for the $2^k$-rank of the narrow class group $\Cl(\Q(\sqrt{dp}))$ of $\Q(\sqrt{dp})$, i.e., there exists a finite normal extension $M_{d, k}/\Q$ and a class function
$$
\phi_{d, k}:\ \Gal(M_{d, k}/\Q)\rightarrow \Z_{\geq 0}
$$ 
such that 
\begin{equation}\label{govEq}
\phi_{d, k}(\mathrm{Art}_{M_{d, k}/\Q}(p)) = \rk_{2^k}\Cl(\Q(\sqrt{dp})),
\end{equation}
where $\mathrm{Art}_{M_{d, k}/\Q}(p)$ is the Artin conjugacy class of $p$ in $\Gal(M_{d, k}/\Q)$. This conjecture was proven for all $k \leq 3$ by Stevenhagen \cite{Stevenhagen}, but no governing field has been found for any value of $d$ if $k \geq 4$. Interestingly enough, Smith \cite{Smith} recently introduced the notion of relative governing fields and used them to deal with distributional questions for $\Cl(K)[2^\infty]$ for imaginary quadratic fields $K$. Our next theorem, which we will prove in Section~\ref{noGovField}, is a relatively straightforward consequence of Theorem \ref{tJoint}.

\begin{theorem}
Assume conjecture $C_n$ for all $n$. Then there is no governing field for the $16$-rank of $\Q(\sqrt{-4p})$; in other words, there does not exist a field $M_{-4, 4}$ and class function $\phi_{-4, 4}$ satisfying~\eqref{govEq}.
\end{theorem}

\subsection*{Acknowledgments}
The authors are very grateful to Carlo Pagano for useful discussions. We would also like to thank Peter Sarnak for making us aware of the useful reference \cite{BhargavaThegeometricsieve}.

\section{Prerequisites}
Here we collect certain facts about quadratic residue symbols and unit groups in number fields that are necessary to give a rigorous definition of spins of ideals and that are useful in our subsequent arguments.

Throughout this section, let $K$ be a number field which is Galois of degree $n$ over $\Q$. Then either $K$ is totally real, as in \cite{FIMR}, or $K$ is totally complex, in which case $n$ is even. An element $\alpha\in K$ is called \textit{totally positive} if $\iota(\alpha)>0$ for all real embeddings $\iota: K\hookrightarrow \mathbb{R}$; if this is the case, we will write $\alpha\succ 0$. If $K$ is totally complex, there are no real embeddings of $K$ into $\mathbb{R}$, and so $\alpha\succ 0$ for every $\alpha\in K$ vacuously. Let $\OO_K$ denote the ring of integers of $K$. If $K$ is totally real, we assume that
\begin{equation}\label{unitcondition}
(\OO_K^{\times})^{2} = \left\{u^2: u\in\OO_K^{\times}\right\} = \left\{u\in\OO_K^{\times}: u\succ 0\right\} = (\OO_K^{\times})_+,
\end{equation}
where the first and last equalities are definitions and the middle equality is the assumption. This assumption, present in \cite{FIMR}, implies that the narrow and the ordinary class groups of $K$ coincide, and hence that every non-zero principal ideal $\aaa$ in $\OO_K$ can be written as $\aaa = \alpha\OO_K$ for some $\alpha\succ 0$. If $K$ is totally complex, then the narrow and the ordinary class groups of $K$ coincide vacuously. In either case, we will let $\Cl = \Cl(K)$ and $h = h(K)$ denote the (narrow) class group and the (narrow) class number of $K$.
 
\subsection{Quadratic residue symbols and quadratic reciprocity}
We define the quadratic residue symbol in $K$ in the standard way. That is, given an odd prime ideal $\pp$ of $\OO_K$ (i.e., a prime ideal having odd absolute norm), and an element $\alpha\in\OO_K$, define $\left(\frac{\alpha}{\pp}\right)_{K, 2}$ as the unique element in $\{-1, 0, 1\}$ such that 
$$
\left(\frac{\alpha}{\pp}\right)_{K, 2} \equiv \alpha^{\frac{\Norm_{K/\Q}(\pp) - 1}{2}} \bmod \pp.
$$
Given an odd ideal $\bbb$ of $\OO_K$ with prime ideal factorization $\bbb = \prod_{\pp}\pp^{e_{\pp}}$, define
$$
\left(\frac{\alpha}{\bbb}\right)_{K, 2} = \prod_{\pp}\left(\frac{\alpha}{\pp}\right)_{K, 2}^{e_{\pp}}.
$$
Finally, given an element $\beta\in\OO_K$, let $(\beta)$ denote the principal ideal in $\OO_K$ generated by $\beta$. We say that $\beta$ is odd if $(\beta)$ is odd and we define
$$
\left(\frac{\alpha}{\beta}\right)_{K, 2} = \left(\frac{\alpha}{(\beta)}\right)_{K, 2}.
$$
We will suppress the subscripts $K, 2$ when there is no risk of ambiguity. Although \cite{FIMR} focuses on a special type of totally real Galois number fields, the version of quadratic reciprocity stated in \cite[Section 3]{FIMR} holds and was proved for a general number field. We recall it here. For a place $v$ of $K$, finite or infinite, let $K_v$ denote the completion of $K$ with respect to $v$. Let $(\cdot, \cdot)_v$ denote the Hilbert symbol at $v$, i.e., given $\alpha, \beta\in K$, we let $(\alpha, \beta)_v\in\{-1, 1\}$ with $(\alpha, \beta)_v = 1$ if and only if there exists $(x, y, z)\in K_v^3\setminus\{(0, 0, 0)\}$ such that $x^2 - \alpha y^2 - \beta z^2 = 0$. As in \cite[Section 3]{FIMR}, define
$$
\mu_2(\alpha, \beta) = \prod_{v\mid 2}(\alpha, \beta)_v\quad\text{and}\quad\mu_{\infty}(\alpha, \beta) = \prod_{v\mid\infty}(\alpha, \beta)_v.
$$
The following lemma is a consequence of the Hilbert reciprocity law and local considerations at places above $2$; see \cite[Lemma 2.1, Proposition 2.2, and Lemma 2.3]{FIMR}. 
\begin{lemma}
Let $\alpha, \beta\in\OO_K$ with $\beta$ odd. Then $\mu_{\infty}(\alpha, \beta)\left(\frac{\alpha}{\beta}\right)$ depends only on the congruence class of $\beta$ modulo $8\alpha$. Moreover, if $\alpha$ is also odd, then 
$$
\left(\frac{\alpha}{\beta}\right) = \mu_{2}(\alpha, \beta)\mu_{\infty}(\alpha, \beta)\left(\frac{\beta}{\alpha}\right).
$$
The factor $\mu_2(\alpha, \beta)$ depends only on the congruence classes of $\alpha$ and $\beta$ modulo $8$.
\end{lemma}
We remark that if $K$ is totally complex, then  $(\alpha, \beta)_{\infty} = 1$ for all $\alpha, \beta\in K$. Also, if $K$ is a totally real Galois number field and $\beta\in K$ is totally positive, then again $(\alpha, \beta)_{\infty} = 1$ for all $\alpha \in K$.

\subsection{Class group representatives}
As in \cite[p. 707]{FIMR}, we define a set of ideals $\Cl$ and an ideal $\ff$ of $\OO_K$ as follows. Let $C_i$, $1\leq i \leq h$, denote the $h$ ideal classes. For each $i\in\{1, \ldots, h\}$, we choose two distinct odd ideals belonging to $C_i$, say $\AAA_i$ and $\BBB_i$, so as to ensure that, upon setting
$$
\Cl_a = \{\AAA_1, \ldots, \AAA_h\},\quad \Cl_b = \{\BBB_1, \ldots, \BBB_h\},\quad\Cl = \Cl_a\cup\Cl_b,
$$
and
$$
\ff = \prod_{\ccc\in\Cl}\ccc = \prod_{i = 1}^h\AAA_i\BBB_i,
$$
the norm 
$$
f = \Norm(\ff)
$$
is squarefree. We define
\begin{equation}
\label{bigF}
F := 2^{2h + 3} f D_K,
\end{equation}
where $D_K$ is the discriminant of $K$.
 
\subsection{Definition of joint spin}
\label{sDefSpin}
We define a sequence $\{ s_{\aaa} \}_{\aaa}$ of complex numbers indexed by non-zero ideals $\aaa\subset\OO_K$ as follows. Let $S$ be a non-empty subset of $\Gal(K/\Q)$ such that $\sigma\not\in S$ whenever $\sigma^{-1}\in S$. We define $r(\aaa)$ to be the indicator function of an ideal $\aaa$ of $\OO_K$ to be odd and principal, i.e.,
$$
r(\aaa) = 
\begin{cases}
1 & \text{if there exists an odd }\alpha\in\OO_K\text{ such that }\aaa = \alpha\OO_K \\
0 & \text{otherwise.}
\end{cases}
$$
Define $r_+(\alpha)$ to be the indicator function of an element $\alpha\in K$ to be totally positive, i.e.,
$$
r_+(\alpha) = 
\begin{cases}
1 & \text{if }\alpha\succ 0\\
0 & \text{otherwise.}
\end{cases}
$$
Note that if $K$ is a totally complex number field, then vacuously $r_+(\alpha) = 1$ for all $\alpha$ in $K$. If $\alpha\in K$ is odd and $r_+(\alpha) = 1$, then we define
$$
\spin(\sigma, \alpha) = \left(\frac{\alpha}{\sigma(\alpha)}\right).
$$
Fix a decomposition $\OO_K^{\times} = T_K\times V_K$, where $T_K\subset \OO_K^{\times}$ is the group of units of $\OO_K$ of finite order and $V_K\subset \OO_K^{\times}$ is a free abelian group of rank $r_K$ (i.e., $r_K = n-1$ if $K$ is totally real and $r_K = \frac{n}{2}-1$ if $K$ is totally complex). With $F$ as in \eqref{bigF}, suppose that
\begin{equation}\label{veModF}
\psi: (\OO_K/F\OO_K)^{\times}\rightarrow \mathbb{C}
\end{equation}
is a map such that $\psi(\alpha \bmod F) = \psi(\alpha u^2 \bmod F)$ for all $\alpha\in\OO_K$ coprime to $F$ and all $u\in\OO_K^{\times}$. We define
\begin{equation}\label{defSpin}
s_{\aaa} = r(\aaa)\sum_{t\in T_K}\sum_{v\in V_K/V_K^2}r_+(tv\alpha)\psi(tv\alpha\bmod F)\prod_{\sigma\in S}\spin(\sigma, tv\alpha),
\end{equation}
where $\alpha$ is any generator of the ideal $\aaa$ satisfying $r(\aaa) = 1$. The averaging over $V_K/V_K^2$ makes the \textit{spin} $s_{\aaa}$ a well-defined function of $\aaa$ since, for any unit $u\in \OO_K^{\times}$, any totally positive $\alpha\in\OO_K$ of odd absolute norm, and any $\sigma\in S$, we have
$$
\spin(\sigma, u^2\alpha) = \left(\frac{u^2\alpha}{\sigma(u^2\alpha)}\right) = \left(\frac{u^2\alpha}{\sigma(\alpha)}\right) = \left(\frac{\alpha}{\sigma(\alpha)}\right) = \spin(\sigma, \alpha).
$$
If $K$ is a totally real (in which case we assume that $K$ satisfies \eqref{unitcondition}), then, for an ideal $\aaa = \alpha\OO_K$, there is one and only one choice of $t\in T_K$ and $v\in V_K/V_K^2$ such that $r_+(tv\alpha) = 1$. Hence in this case
$$
s_{\aaa} = r(\aaa)\psi(\alpha\bmod F)\prod_{\sigma\in S}\spin(\sigma, \alpha),
$$
where $\alpha$ is any totally positive generator of $\aaa$. If in addition $n\geq 3$, $\Gal(K/\Q) = \langle\sigma\rangle$, and $S = \{\sigma\}$, then $s_{\aaa}$ coincides with $\spin(\sigma, \aaa)$ in \cite[(3.4), p. 706]{FIMR}. If we take instead $S = \{\sigma, \sigma^2\}$ and assume $n\geq 5$, then the distribution of $s_{\aaa}$ has implications for \cite[Problem, p. 744]{FIMR}.

If $K$ is totally complex, then vacuously $r_+(tv\alpha) = 1$ for all $t\in T_K$ and $v\in V_K/V_K^2$, so the definition of $s_{\aaa}$ specializes to
$$
s_{\aaa} = r(\aaa)\sum_{t\in T_K}\sum_{v\in V_K/V_K^2}\psi(tv\alpha\bmod F)\prod_{\sigma\in S}\spin(\sigma, tv\alpha).
$$

\subsection{Fundamental domains}
\label{fundDom}
We will need a suitable fundamental domain $\DD$ for the action of the units on elements in $\OO_K$.

In case that $K$ is totally real and satisfies \eqref{unitcondition}, we take $\DD\subset \mathbb{R}_+^n$ to be the same as in \cite[(4.2), p. 713]{FIMR}. We fix a numbering of the $n$ real embeddings $\iota_1, \ldots, \iota_n: K\hookrightarrow \mathbb{R}$, and we say that $\alpha\in\DD$ if and only if $(\iota_1(\alpha), \ldots, \iota_n(\alpha))\in\DD$. Hence every non-zero $\alpha\in\DD$ is totally positive. Because of the assumption \eqref{unitcondition}, every non-zero principal ideal in $\OO_K$ has a totally positive generator, and $\DD$ is a fundamental domain for the action of $(\OO_K)_+^{\times}$ on the totally positive elements in $\OO_K$, in the sense of \cite[Lemma 4.3, p. 715]{FIMR}. 

In case that $K$ is totally complex, we take $\DD\subset \mathbb{R}^n$ to be the same as in \cite[Lemma 3.5, p. 10]{KM1}. In this case, we fix an integral basis $\{\eta_1, \ldots, \eta_n\}$ for $\OO_K$, and if $\alpha = a_1\eta_1+\cdots+a_n\eta_n \in K$, $a_1, \ldots, a_n\in\Q$, we say that $\alpha\in\DD$ if and only if $(a_1, \ldots, a_n)\in\DD$. Every non-zero principal ideal $\aaa$ in $\OO_K$ has exactly $|T_K|$ generators in $\DD$; moreover, if one of the generators of $\aaa$ in $\DD$ is $\alpha$, say, then the set of generators of $\aaa$ in $\DD$ is $\{t\alpha: t\in T_K\}$.

The main properties of $\DD$ are listed in \cite[Lemma 4.3, Lemma 4.4, Corollary 4.5]{FIMR} and \cite[Lemma 3.5]{KM2}. We will often use the property that if an element $\alpha\in\DD\cap\OO_K$ of norm $\Norm(\alpha)\leq X$ is written in an integral basis $\eta = \{\eta_1, \ldots, \eta_n\}$ as $\alpha = a_1\eta_1+\cdots+a_n\eta_n \in \OO_K$, $a_1, \ldots, a_n\in\Z$, then
$$
|a_i|\ll X^{\frac{1}{n}}
$$
for $1\leq i\leq n$ where the implied constant depends only on $\eta$.

\subsection{Short character sums}
\label{conjCn}
The following is a conjecture on short character sums appearing in \cite{FIMR}. It is essential for the estimates for sums of type I. 
\begin{conjecture}
\label{cn}
For all integers $n \geq 3$ there exists $\delta(n) > 0$ such that for all $\epsilon > 0$ there exists a constant $C(n, \epsilon) > 0$ with the property that for all integers $M$, all integers $Q \geq 3$, all integers $N \leq Q^{\frac{1}{n}}$ and all real non-principal characters $\chi$ of modulus $q \leq Q$ we have
\[
\left|\sum_{M < m \leq M + N} \chi(m)\right| \leq C(n, \epsilon) Q^{\frac{1 - \delta(n)}{n} + \epsilon}.
\]
\end{conjecture}

Instead of working directly with Conjecture $C_n$, we need a version of it for arithmetic progressions. If $q$ is odd and squarefree, we let $\chi_q$ be the real Dirichlet character $\left(\frac{\cdot}{q}\right)$.

\begin{corollary}
\label{cscs}
Assume Conjecture $C_n$. Then for all integers $n \geq 3$ there exists $\delta(n) > 0$ such that for all $\epsilon > 0$ there exists a constant $C(n, \epsilon) > 0$ with the property that for all odd squarefree integers $q > 1$, all integers $N \leq q^{\frac{1}{n}}$, all integers $M$, $l$ and $k$ with $q \nmid k$, we have
\[
\left|\sum_{\substack{M < m \leq M + N \\ n \equiv l \bmod k}} \chi_q(m) \right| \leq C(n, \epsilon) q^{\frac{1 - \delta(n)}{n}}.
\]
\end{corollary}

\begin{proof}
This is an easy generalization of Corollary 7 in \cite{KM1}.
\end{proof}

\subsection{The sieve}
We will prove the following oscillation results for the sequence $\{ s_{\aaa} \}_{\aaa}$. First, for any non-zero ideal $\mm\subset \OO_K$ and any $\epsilon > 0$, we have 
\begin{equation}\label{typeI}
\sum_{\substack{\Norm(\aaa)\leq X\\ \aaa\equiv 0\bmod \mm}}s_{\aaa} \ll_{\epsilon} X^{1-\frac{\delta}{54n|S|^2} + \epsilon},
\end{equation}
where $\delta$ is as in Conjecture $C_n$. Second, for any $\epsilon>0$, we have
\begin{equation}\label{typeII}
\sum_{\Norm(\aaa)\leq x}\sum_{\Norm(\bb)\leq y}v_{\aaa}w_{\bb}s_{\aaa\bb} \ll_{\epsilon} \left(x^{-\frac{1}{6n}}+y^{-\frac{1}{6n}}\right)\left(xy\right)^{1+\epsilon},
\end{equation}
for any pair of bounded sequences of complex numbers $\{v_{\mm}\}$ and $\{w_{\nn}\}$ indexed by non-zero ideals in $\OO_K$. Then \cite[Proposition 5.2, p. 722]{FIMR} implies that for any $\epsilon>0$, we have
$$
\sum_{\substack{\Norm(\pp)\leq X\\ \pp\text{ prime ideal}}}s_{\pp} \ll_{\epsilon} X^{1-\theta + \epsilon},
$$
where
$$
\theta := \frac{\delta(|S|n)}{54|S|^2n(12n+1)}.
$$
Hence, in order to prove Theorem \ref{tJoint}, it suffices to prove the estimates (\ref{typeI}) and (\ref{typeII}). We will deal with (\ref{typeI}) in Section \ref{sLinear} and with (\ref{typeII}) in Section \ref{sBilinear}.

\section{Linear sums}
\label{sLinear}
We first treat the case that $K$ is totally real. Let $\mathfrak{m}$ be an ideal coprime with $F$ and $\sigma(\mathfrak{m})$ for all $\sigma \in S$. Following \cite{FIMR} we will bound
\begin{align}
\label{eFirstSum}
A(x) = \sum_{\substack{\Norm{\mathfrak{a}} \leq x \\ (\mathfrak{a}, F) = 1, \mathfrak{m} \mid \mathfrak{a}}} r(\mathfrak{a}) \psi(\alpha \bmod F) \prod_{\sigma \in S} \spin(\sigma, \alpha),
\end{align}
where $\alpha$ is any totally positive generator of $\mathfrak{a}$. We pick for each ideal $\mathfrak{a}$ with $r(\mathfrak{a}) = 1$ its unique generator $\alpha$ satisfying $\mathfrak{a} = (\alpha)$ and $\alpha \in \DD^\ast$, where $\DD^\ast$ is the fundamental domain from Friedlander et al. \cite{FIMR}. After splitting (\ref{eFirstSum}) in residue classes modulo $F$ we obtain
\begin{align*}
A(x) = \sum_{\substack{\rho \bmod F \\ (\rho, F) = 1}} \psi(\rho) A(x; \rho) + \partial A(x),
\end{align*}
where by definition
\begin{align}
\label{eArho}
A(x; \rho) := \sum_{\substack{\alpha \in \DD, \Norm{\alpha} \leq x \\ \alpha \equiv \rho \bmod F \\ \alpha \equiv 0 \bmod \mathfrak{m}}} \prod_{\sigma \in S} \spin(\sigma, \alpha).
\end{align}
The boundary term $\partial A(x)$ can be dealt with using the argument in \cite[p.\ 724]{FIMR}, which gives $\partial A(x) \ll x^{1 - \frac{1}{n}}$. Here and in the rest of our arguments the implied constant depends only on $K$ unless otherwise indicated. We will now estimate $A(x; \rho)$ for each $\rho \bmod F$, $(\rho, F) = 1$. Let $1, \omega_2, \ldots, \omega_n$ be an integral basis for $\OO_K$ and define
\begin{align*}
\MM := \omega_2 \Z + \cdots + \omega_n \Z.
\end{align*}
Then, just as in \cite[p.\ 725]{FIMR}, we can decompose $\alpha$ uniquely as 
\begin{align*}
\alpha = a + \beta, \quad \quad \text{with } a \in \Z, \beta \in \MM.
\end{align*}
Hence the summation conditions in (\ref{eArho}) can be rewritten as
\begin{equation}\tag{$\ast$}
a + \beta \in \DD, \quad \Norm{(a + \beta)} \leq x, \quad a + \beta \equiv \rho \bmod F, \quad a + \beta \equiv 0 \bmod \mathfrak{m}.
\end{equation}
From now on we think of $a$ as a variable satisfying ($\ast$) while $\beta$ is inactive. We have the following formula
\begin{align*}
\spin(\sigma, \alpha) = \left(\frac{\alpha}{\sigma(\alpha)}\right) = \left(\frac{a + \beta}{a + \sigma(\beta)}\right) = \left(\frac{\beta - \sigma(\beta)}{a + \sigma(\beta)}\right).
\end{align*}
If $\beta = \sigma(\beta)$ for some $\sigma \in S$ we get no contribution. So from now on we can assume $\beta \neq \sigma(\beta)$ for all $\sigma \in S$. Define $\ccc(\sigma, \beta)$ to be the part of the ideal $(\beta - \sigma(\beta))$ coprime to $F$. Then, as explained on \cite[p.\ 726]{FIMR}, quadratic reciprocity gives
\begin{align*}
A(x; \rho) = \sum_{\beta \in \MM} \pm T(x; \rho, \beta),
\end{align*}
where $T(x; \rho, \beta)$ is given by
\begin{align}
T(x; \rho, \beta) &:= \sum_{\substack{a \in \Z \\ a + \beta \text{ sat. }(\ast)}} \prod_{\sigma \in S} \left(\frac{a + \sigma(\beta)}{\ccc(\sigma, \beta)}\right) = \sum_{\substack{a \in \Z \\ a + \beta \text{ sat. }(\ast)}} \prod_{\sigma \in S} \left(\frac{a + \beta}{\ccc(\sigma, \beta)}\right) \nonumber \\
\label{eT}
&= \sum_{\substack{a \in \Z \\ a + \beta \text{ sat. }(\ast)}} \left(\frac{a + \beta}{\prod_{\sigma \in S} \ccc(\sigma, \beta)}\right).
\end{align}
Define $\ccc := \prod_{\sigma \in S} \ccc(\sigma, \beta)$ and factor $\ccc$ as
\begin{align}
\label{eFactorc}
\ccc = \mathfrak{g} \qq,
\end{align}
where by definition $\mathfrak{g}$ consists of those prime ideals $\pp$ dividing $\ccc$ that satisfy one of the following three properties
\begin{itemize}
\item $\pp$ has degree greater than one;
\item $\pp$ is unramified of degree one and some non-trivial conjugate of $\pp$ also divides $\ccc$;
\item $\pp$ is unramified of degree one and $\pp^2$ divides $\ccc$.
\end{itemize}
Note that there are no ramified primes dividing $\ccc$, since $\ccc$ is coprime to the discriminant by construction of $F$. Putting all the remaining prime ideals in $\qq$, we note that $q := \Norm{\mathfrak{q}}$ is a squarefree number and $g := \Norm{\mathfrak{g}}$ is a squarefull number coprime with $q$. The Chinese Remainder Theorem implies that there exists a rational integer $b$ with $b \equiv \beta \bmod \mathfrak{q}$. We stress that $\ccc$, $\mathfrak{g}$, $\qq$, $g$, $q$ and $b$ depend only on $\beta$. Define $g_0$ to be the radical of $g$. Then the quadratic residue symbol $(\alpha/\mathfrak{g})$ is periodic in $\alpha$ modulo $g_0$. Hence the symbol $((a + \beta)/\mathfrak{g})$ as a function of $a$ is periodic of period $g_0$. Splitting the sum (\ref{eT}) in residue classes modulo $g_0$ we obtain
\begin{align}
\label{eT2}
|T(x; \rho, \beta)| \leq \sum_{a_0 \bmod g_0} \left| \sum_{\substack{a \equiv a_0 \bmod g_0 \\ a + \beta \text{ sat. }(\ast)}} \left(\frac{a + b}{\qq}\right) \right|.
\end{align}
Following the argument on \cite[p.\ 728]{FIMR}, we see that (\ref{eT2}) can be written as $n$ incomplete character sums of length $\ll x^{\frac{1}{n}}$ and modulus $q \ll x^{|S|}$. Furthermore, the conditions ($\ast$) and $a \equiv a_0 \bmod g_0$ imply that $a$ runs over a certain arithmetic progression of modulus $k$ dividing $g_0Fm$, where $m := \Norm{\mm}$. So if $q \nmid k$, Corollary \ref{cscs} yields
\begin{align}
\label{eFirstBound}
T(x; \rho, \beta) \ll_\epsilon g_0 x^{\frac{1 - \delta}{n} + \epsilon}
\end{align}
with $\delta := \delta(|S|n) > 0$. Since $q \mid k$ implies $q \mid m$, we see that (\ref{eFirstBound}) holds if $q \nmid m$. Recalling (\ref{eFactorc}) we conclude that (\ref{eFirstBound}) holds unless
\begin{align}
\label{eException}
p \mid \prod_{\sigma \in S} \Norm{(\beta - \sigma(\beta))} \Rightarrow p^2 \mid mF\prod_{\sigma \in S} \Norm{(\beta - \sigma(\beta))}.
\end{align}
Our next goal is to count the number of $\beta \in \MM$ satisfying both $(\ast)$ for some $a \in \Z$ and (\ref{eException}). For $\beta$ an algebraic integer of degree $n$, we denote by $\beta^{(1)}, \ldots, \beta^{(n)}$ the conjugates of $\beta$. Now if $\beta$ satisfies $(\ast)$ for some $a \in \Z$, we have $|\beta^{(i)}| \ll x^{\frac{1}{n}}$. So to achieve our goal, it suffices to estimate the number of $\beta \in \MM$ satisfying $|\beta^{(i)}| \leq x^{\frac{1}{n}}$ and (\ref{eException}). 

To do this, we will need two lemmas. So far we have followed \cite{FIMR} rather closely, but we will have to significantly improve their estimates for the various error terms given on \cite[p.\ 729-733]{FIMR}. One of the most important tasks ahead is to count squarefull norms in a certain $\Z$-submodule of $\OO_K$. This problem is solved in \cite{FIMR} by simply counting squarefull norms in the full ring of integers. For our application this loss is unacceptable. In our first lemma we directly count squarefull norms in this submodule, a problem described in \cite[p.\ 729]{FIMR} as potentially ``very difficult''.

\begin{lemma}
\label{lG}
Factor $\ccc(\sigma, \beta)$ as
\begin{align*}
\ccc(\sigma, \beta) = \mathfrak{g}(\sigma, \beta) \qq(\sigma, \beta)
\end{align*}
just as in (\ref{eFactorc}). Let $K^{\sigma}$ be the subfield of $K$ fixed by $\sigma$ and let $\OO_{K^\sigma}$ be its ring of integers. Decompose $\OO_K$ as
\begin{align*}
\OO_K = \OO_{K^\sigma} \oplus \MM'.
\end{align*}
Let $\text{ord}(\sigma)$ be the order of $\sigma$ in $\Gal(K/\Q)$. If $g_0(\sigma, \beta)$ is the radical of $\Norm{\mathfrak{g}(\sigma, \beta)}$, then we have for all $\epsilon > 0$
\begin{align*}
|\{\beta \in \MM': |\beta^{(i)}| \leq x^{\frac{1}{n}}, g_0(\sigma, \beta) > Z\}| \ll_\epsilon x^{1 - \frac{1}{\text{ord}(\sigma)} + \epsilon} Z^{-1 + \frac{2}{\text{ord}(\sigma)}}.
\end{align*}
\end{lemma}

\begin{proof}
The argument given here is a generalization of \cite[p.\ 17-18]{KM1}. We start with the simple estimate
\begin{align}
\label{eA}
|\{\beta \in \MM': |\beta^{(i)}| \leq x^{\frac{1}{n}}, g_0(\sigma, \beta) > Z\}| \leq \sum_{\substack{\mathfrak{g} \\ g_0 > Z}} A_\mathfrak{g},
\end{align}
where
\begin{align*}
A_\mathfrak{g} := |\{\beta \in \MM' : |\beta^{(i)}| \leq x^{\frac{1}{n}}, \beta - \sigma(\beta) \equiv 0 \bmod \mathfrak{g}\}|.
\end{align*}
Let $\MM''$ be the image of $\MM'$ under the map $\beta \mapsto \beta - \sigma(\beta)$ and fix a $\Z$-basis $\eta_1, \ldots, \eta_r$ of $\MM''$. We remark that $r = n\left(1 - \frac{1}{\text{ord}(\sigma)}\right)$, which will be important later on. Because $|\beta^{(i)}| \leq x^{\frac{1}{n}}$, we can write $\beta - \sigma(\beta)$ as $\beta - \sigma(\beta) = \sum_{i = 1}^r a_i \eta_i$ with $|a_i| \leq C_K x^{\frac{1}{n}}$, where $C_K$ is a constant depending only on $K$. Hence we have 
\begin{align*}
A_\mathfrak{g} \leq |\Lambda_\mathfrak{g} \cap S_x|,
\end{align*}
where by definition
\begin{align*}
\Lambda_\mathfrak{g} &:= \{\gamma \in \MM'' : \gamma \equiv 0 \bmod \mathfrak{g}\} \\
S_x &:= \{\gamma \in \MM'' : \gamma = \sum_{i = 1}^r a_i \eta_i, |a_i| \leq C_K x^{\frac{1}{n}}\}.
\end{align*}
Using our fixed $\Z$-basis $\eta_1, \ldots, \eta_r$ we can view $\MM''$ as a subset of $\mathbb{R}^r$ via the map $\eta_i \mapsto e_i$, where $e_i$ is the $i$-th standard basis vector. Under this identification $\MM''$ becomes $\Z^r$ and $\Lambda_\mathfrak{g}$ becomes a sublattice of $\Z^r$. We have
\begin{align}
\label{eAg}
A_\mathfrak{g} \leq |\Lambda_\mathfrak{g} \cap T_x|,
\end{align}
where
\begin{align*}
T_x &:= \{(a_1, \ldots, a_r) \in \mathbb{R}^r : |a_i| \leq C_K x^{\frac{1}{n}}\}.
\end{align*}
Let us now parametrize the boundary of $T_x$. We start off by observing that $T_x = x^{\frac{1}{n}}T_1$, which implies that $\Vol(T_x) = x^\frac{r}{n} \Vol(T_1)$. Because $T_1$ is an $r$-dimensional hypercube, we conclude that its boundary $\partial T_1$ can be parametrized by Lipschitz functions with Lipschitz constant $L$ depending only on $K$. Therefore $\partial T_x$ can also be parametrized by Lipschitz functions with Lipschitz constant $x^{\frac{1}{n}}L$. Theorem 5.4 of \cite{Widmer} gives
\begin{align}
\label{eLattice}
\left||\Lambda_\mathfrak{g} \cap T_x| - \frac{\Vol(T_x)}{\det \Lambda_\mathfrak{g}}\right| \ll_L \max_{0 \leq i < r} \frac{x^{\frac{i}{n}}}{\lambda_{\mathfrak{g}, 1} \cdot \ldots \cdot \lambda_{\mathfrak{g}, i}},
\end{align}
where $\lambda_{\mathfrak{g}, 1}, \ldots, \lambda_{\mathfrak{g}, r}$ are the successive minima of $\Lambda_\mathfrak{g}$. Since $L$ depends only on $K$, it follows that the implied constant in (\ref{eLattice}) depends only on $K$, so we may simply write $\ll$ by our earlier conventions.

Our next goal is to give a lower bound for $\lambda_{\mathfrak{g}, 1}$. So let $\gamma \in \Lambda_\mathfrak{g}$ be non-zero. By definition of $\Lambda_\mathfrak{g}$ we have $\mathfrak{g} \mid \gamma$ and hence $g \mid \Norm{\gamma}$. Write
\begin{align*}
\gamma = \sum_{i = 1}^r a_i \eta_i.
\end{align*}
If $a_1, \ldots, a_r \leq C'_K g^{\frac{1}{n}}$ for a sufficiently small constant $C'_K$, we find that $\Norm{\gamma} < g$. But this is impossible, since $g \mid \Norm{\gamma}$ and $\Norm{\gamma} \neq 0$. So there is an $i$ with $a_i > C'_K g^{\frac{1}{n}}$. If we equip $\mathbb{R}^r$ with the standard Euclidean norm, we conclude that the length of $\gamma$ satisfies $||\gamma|| \gg g^{\frac{1}{n}}$ and hence
\begin{align}
\label{eLL}
\lambda_{\mathfrak{g}, 1} \gg g^{\frac{1}{n}}.
\end{align}
Minkowski's second theorem and (\ref{eLL}) imply that
\begin{align}
\label{eLD}
\det \Lambda_\mathfrak{g} \gg g^{\frac{r}{n}}.
\end{align}
Combining (\ref{eLattice}), (\ref{eLL}), (\ref{eLD}) and $g \leq x$ gives
\begin{align}
\label{eI}
|\Lambda_\mathfrak{g} \cap T_x| \ll \frac{x^{\frac{r}{n}}}{g^{\frac{r}{n}}} + \frac{x^{\frac{r - 1}{n}}}{g^{\frac{r - 1}{n}}} \ll \frac{x^{\frac{r}{n}}}{g^{\frac{r}{n}}}.
\end{align}
Plugging (\ref{eAg}) and (\ref{eI}) back in (\ref{eA}) yields
\begin{align*}
|\{\beta \in \MM': |\beta^{(i)}| \leq x^{\frac{1}{n}}, g_0(\sigma, \beta) > Z\}| \leq \sum_{\substack{\mathfrak{g} \\ g_0 > Z}} A_\mathfrak{g} \leq \sum_{\substack{\mathfrak{g} \\ g_0 > Z}} |\Lambda_\mathfrak{g} \cap T_x| \ll \sum_{\substack{\mathfrak{g} \\ g_0 > Z}} \frac{x^{\frac{r}{n}}}{g^{\frac{r}{n}}}.
\end{align*}
If we define $\tau_K(g)$ to be the number of ideals of $K$ of norm $g$, we can bound the last sum as follows
\begin{align*}
\sum_{\substack{\mathfrak{g} \\ g_0 > Z}} \frac{x^{\frac{r}{n}}}{g^{\frac{r}{n}}} 
&= x^{\frac{r}{n}} \sum_{\substack{g \leq x \\ g \text{ squarefull} \\ g_0 > Z}} \frac{\tau_K(g)}{g^{\frac{r}{n}}}
\ll_\epsilon x^{\frac{r}{n} + \epsilon} \sum_{\substack{g \leq x \\ g \text{ squarefull} \\ g_0 > Z}} \frac{1}{g^{\frac{r}{n}}} \\
&= x^{\frac{r}{n} + \epsilon} \sum_{\substack{g \leq x \\ g \text{ squarefull} \\ g_0 > Z}} g^{\frac{1}{2} - \frac{r}{n}} \frac{1}{g^{\frac{1}{2}}}
\leq x^{\frac{r}{n} + \epsilon} Z^{1 - \frac{2r}{n}} \sum_{\substack{g \leq x \\ g \text{ squarefull} \\ g_0 > Z}} \frac{1}{g^{\frac{1}{2}}} \\
&\leq x^{\frac{r}{n} + \epsilon} Z^{1 - \frac{2r}{n}} \sum_{\substack{g \leq x \\ g \text{ squarefull}}} \frac{1}{g^{\frac{1}{2}}} \ll_\epsilon x^{\frac{r}{n} + \epsilon} Z^{1 - \frac{2r}{n}}.
\end{align*}
Recalling that $r = n\left(1 - \frac{1}{\text{ord}(\sigma)}\right)$ completes the proof of Lemma \ref{lG}.
\end{proof}

\begin{lemma}
\label{lGCD}
Let $\sigma, \tau \in S$ be distinct. Recall that
\begin{align*}
\OO_K = \Z \oplus \MM.
\end{align*}
Fix an integral basis $\omega_2, \ldots, \omega_n$ of $\MM$ and define the polynomials $f_1, f_2 \in \Z[x_2, \ldots, x_n]$ by
\begin{align*}
f_1(x_2, \ldots, x_n) &= \Norm{\left(\sum_{i = 2}^n x_i (\sigma(\omega_i) - \omega_i)\right)} \\
f_2(x_2, \ldots, x_n) &= \Norm{\left(\sum_{i = 2}^n x_i (\tau(\omega_i) - \omega_i)\right)}.
\end{align*}
For $\beta \in \MM$ with $\beta = \sum_{i = 2}^n a_i \omega_i$ we define $f_1(\beta) := f_1(a_2, \ldots, a_n) = \Norm{(\sigma(\beta) - \beta)}$ and similarly for $f_2(\beta)$. Then
\begin{align*}
|\{\beta \in \MM: |\beta^{(i)}| \leq x^{\frac{1}{n}}, \gcd(f_1(\beta), f_2(\beta)) > Z\}| \ll_\epsilon x^{\frac{n - 1}{n} + \epsilon} Z^{-\frac{1}{18}} + x^{\frac{n - 2}{n}} + Z^{\frac{2n - 4}{3}}.
\end{align*}
\end{lemma}

\begin{proof}
Let $Y$ be the closed subscheme of $\mathbb{A}_\Z^{n - 1}$ defined by $f_1 = f_2 = 0$. We claim that $Y$ has codimension $2$, i.e. $f_1$ and $f_2$ are relatively prime polynomials. Suppose not. Note that $f_1$ and $f_2$ factor in $K[x_2, \ldots, x_n]$ as
\begin{align*}
f_1(x_2, \ldots, x_n) &= \prod_{\sigma' \in \Gal(K/\Q)} \left(\sum_{i = 2}^n x_i (\sigma' \sigma(\omega_i) - \sigma'(\omega_i))\right) \\
f_2(x_2, \ldots, x_n) &= \prod_{\tau' \in \Gal(K/\Q)} \left(\sum_{i = 2}^n x_i (\tau' \tau(\omega_i) - \tau'(\omega_i))\right).
\end{align*}
Hence if $f_1$ and $f_2$ are not relatively prime, there are $\sigma', \tau' \in \Gal(K/\Q)$ and $\kappa \in K^\ast$ such that
\begin{align*}
\sum_{i = 2}^n x_i (\sigma' \sigma(\omega_i) - \sigma'(\omega_i)) = \kappa \sum_{i = 2}^n x_i (\tau' \tau(\omega_i) - \tau'(\omega_i))
\end{align*}
for all $x_2, \ldots, x_n \in \Z$. Put $\beta = \sum_{i = 2}^n x_i \omega_i$. Then we can rewrite this as
\begin{align}
\label{eAD}
\sigma' \sigma(\beta) - \sigma'(\beta) = \kappa(\tau' \tau(\beta) - \tau'(\beta))
\end{align}
for all $\beta \in \MM$. But this implies that (\ref{eAD}) holds for all $\beta \in K$. Now we apply the Artin-Dedekind Lemma, which gives a contradiction in all cases due to our assumptions $\sigma, \tau \in S$ and $\sigma \neq \tau$.

Having established our claim, we are in position to apply Theorem 3.3 of \cite{BhargavaThegeometricsieve}. We embed $\MM$ in $\mathbb{R}^{n - 1}$ by sending $\omega_i$ to $e_i$, the $i$-th standard basis vector. Note that the image under this embedding is $\Z^{n - 1}$. Write $\beta = \sum_{i = 2}^n a_i \omega_i$. Since $|\beta^{(i)}| \leq x^{\frac{1}{n}}$, it follows that $|a_i| \leq C_K x^{\frac{1}{n}}$ for some constant $C_K$ depending only on $K$. Let $B$ be the compact region in $\mathbb{R}^{n - 1}$ given by $B := \{(a_2, \ldots, a_n) : |a_i| \leq C_K\}$. Theorem 3.3 of \cite{BhargavaThegeometricsieve} with our $B$, $Y$ and $r = x^{\frac{1}{n}}$ gives
\begin{align}
\label{eBad1}
|\{\beta \in \MM: |\beta^{(i)}| \leq x^{\frac{1}{n}}, p \mid \gcd(f_1(\beta), f_2(\beta)), p > M\}| \ll \frac{x^{\frac{n - 1}{n}}}{M \log M} + x^{\frac{n - 2}{n}},
\end{align}
where $M$ is any positive real number. Factor
\begin{align*}
f_1(\beta) &:= g_1 q_1, \quad (g_1, q_1) = 1, \quad g_1 \text{ squarefull}, \quad q_1 \text{ squarefree} \\
f_2(\beta) &:= g_2 q_2, \quad (g_2, q_2) = 1, \quad g_2 \text{ squarefull}, \quad q_2 \text{ squarefree}.
\end{align*}
By Lemma \ref{lG} we conclude that for all $A > 0$ and $\epsilon > 0$
\begin{align*}
|\{\beta \in \MM: |\beta^{(i)}| \leq x^{\frac{1}{n}}, g_1 > A\}| \ll_\epsilon x^{\frac{n - 1}{n} + \epsilon} A^{-\frac{1}{2} + \frac{1}{\text{ord}(\sigma)}}.
\end{align*}
With the same argument applied to $\tau$ we obtain
\begin{align}
\label{eBad2}
|\{\beta \in \MM: |\beta^{(i)}| \leq x^{\frac{1}{n}}, g_1 > A \text{ or } g_2 > A\}| \ll_\epsilon x^{\frac{n - 1}{n} + \epsilon} A^{-\frac{1}{2} + \frac{1}{\text{ord}(\sigma)}} + x^{\frac{n - 1}{n} + \epsilon} A^{-\frac{1}{2} + \frac{1}{\text{ord}(\tau)}}.
\end{align}
We discard those $\beta$ that satisfy (\ref{eBad1}) or (\ref{eBad2}). From (\ref{eBad2}) we deduce that the remaining $\beta$ certainly satisfy $\gcd(q_1, q_2) > \frac{Z}{A^2}$. Furthermore, by discarding those $\beta$ satisfying (\ref{eBad1}), we see that $\gcd(q_1, q_2)$ has no prime divisors greater than $M$. This implies that $\gcd(q_1, q_2)$ is divisible by a squarefree number between $\frac{Z}{A^2}$ and $\frac{ZM}{A^2}$. So we must still give an upper bound for
\begin{align}
\label{eTB}
\left|\left\{\beta \in \MM: |\beta^{(i)}| \leq x^{\frac{1}{n}}, r \mid \gcd(q_1, q_2), \frac{Z}{A^2} < r \leq \frac{ZM}{A^2}\right\}\right|.
\end{align}
Let $r$ be a squarefree integer and let $\mathfrak{r}_1, \mathfrak{r}_2$ be two ideals of $K$ with norm $r$. Define
\begin{align*}
E_{\mathfrak{r_1}, \mathfrak{r}_2} := \left|\left\{\beta \in \MM: |\beta^{(i)}| \leq x^{\frac{1}{n}}, \mathfrak{r}_1 \mid \sigma(\beta) - \beta, \mathfrak{r}_2 \mid \tau(\beta) - \beta \right\}\right|.
\end{align*}
We will give an upper bound for $E_{\mathfrak{r_1}, \mathfrak{r}_2}$ following \cite[p.\ 731-733]{FIMR}. Write $\beta = \sum_{i = 2}^n a_i \omega_i$. Then $|\beta^{(i)}| \leq x^{\frac{1}{n}}$ implies $a_i \ll x^{\frac{1}{n}}$ and
\begin{align}
\label{eAL1}
\sum_{i = 2}^n a_i (\sigma(\omega_i) - \omega_i) &\equiv 0 \bmod \mathfrak{r}_1 \\
\label{eAL2}
\sum_{i = 2}^n a_i (\tau(\omega_i) - \omega_i) &\equiv 0 \bmod \mathfrak{r}_2.
\end{align}
We split the coefficients $a_2, \ldots, a_n$ according to their residue classes modulo $r$. Suppose that $p \mid r$ and let $\pp_1$, $\pp_2$ be the unique prime ideals of degree one dividing $\mathfrak{r}_1$ and $\mathfrak{r}_2$ respectively. Then we get
\begin{align}
\label{eLocal1}
\sum_{i = 2}^n a_i (\sigma(\omega_i) - \omega_i) &\equiv 0 \bmod \pp_1 \\
\label{eLocal2}
\sum_{i = 2}^n a_i (\tau' \tau(\omega_i) - \tau'(\omega_i)) &\equiv 0 \bmod \pp_1,
\end{align}
where $\tau'$ satisfies $\tau'^{-1}(\pp_1) = \pp_2$. If we further assume that $\pp_1$ is unramified, we claim that the above two equations are linearly independent over $\mathbb{F}_p$. Indeed, consider the isomorphism
\begin{align*}
\OO_K/p \cong \mathbb{F}_p \times \cdots \times \mathbb{F}_p.
\end{align*}
Note that $\tau' \tau \not \in \{\text{id}, \sigma\}$ or $\tau' \not \in \{\text{id}, \sigma\}$ due to our assumption that $\sigma$ and $\tau$ are distinct elements of $S$. Let us deal with the case $\tau' \tau \not \in \{\text{id}, \sigma\}$, the other case is dealt with similarly. Then there exists $\beta \in \OO_K$ such that $\beta \equiv 1 \bmod \pp_1$, $\beta \equiv 1 \bmod \sigma^{-1}(\pp_1)$, $\beta \equiv 1 \bmod \tau'^{-1}(\pp_1)$ and $\beta$ is divisible by all other conjugates of $\pp_1$. By our assumption on $\tau' \tau$ it follows that $\beta \equiv 0 \bmod \tau^{-1} \tau'^{-1}(\pp_1)$. Hence we obtain
\begin{align*}
\sigma(\beta) - \beta \equiv 0 \bmod \pp_1, \quad \tau' \tau(\beta) - \tau'(\beta) \equiv -1 \bmod \pp_1.
\end{align*}
However, for $\pp_1$ an unramified prime, we know that $\sigma(\beta) - \beta \equiv 0 \bmod \pp_1$ can not happen for all $\beta \in \OO_K$, unless $\sigma$ is the identity. This proves our claim.

If we further split the coefficients $a_2, \ldots, a_n$ according to their residue classes modulo $p$, our claim implies that there are $p^{n - 3}$ solutions $a_2, \ldots, a_n$ modulo $p$ satisfying (\ref{eLocal1}) and (\ref{eLocal2}), provided that $p$ is unramified. For ramified primes we can use the trivial upper bound $p^{n - 1}$. Then we deduce from the Chinese Remainder Theorem that there are $\ll r^{n - 3}$ solutions $a_2, \ldots, a_n$ modulo $r$ satisfying (\ref{eAL1}) and (\ref{eAL2}). This yields
\begin{align*}
E_{\mathfrak{r_1}, \mathfrak{r}_2} \ll r^{n - 3} \left(\frac{x^{\frac{1}{n}}}{r} + 1\right)^{n - 1} \ll x^{\frac{n - 1}{n}} r^{-2} + r^{n - 3}.
\end{align*}
Therefore we have the following upper bound for (\ref{eTB})
\begin{align*}
\sum_{\frac{Z}{A^2} < r \leq \frac{ZM}{A^2}} \sum_{\substack{\mathfrak{r}_1, \mathfrak{r}_2 \\ \Norm{\mathfrak{r}_1} = \Norm{\mathfrak{r}_2} = r}} E_{\mathfrak{r}_1, \mathfrak{r}_2} &\ll \sum_{\frac{Z}{A^2} < r \leq \frac{ZM}{A^2}} \sum_{\substack{\mathfrak{r}_1, \mathfrak{r}_2 \\ \Norm{\mathfrak{r}_1} = \Norm{\mathfrak{r}_2} = r}} x^{\frac{n - 1}{n}} r^{-2} + r^{n - 3} \\
&\ll_{\epsilon} x^{\epsilon} \sum_{\frac{Z}{A^2} < r \leq \frac{ZM}{A^2}} x^{\frac{n - 1}{n}} r^{-2} + r^{n - 3} \\
&\ll_{\epsilon} x^{\epsilon}\left(x^{\frac{n - 1}{n}} \frac{A^2}{Z} + \left(\frac{ZM}{A^2}\right)^{n - 2}\right).
\end{align*}
Note that $\sigma \in S$ implies $\text{ord}(\sigma) \geq 3$. Now choose $A = M = Z^{\frac{1}{3}}$ to complete the proof of Lemma \ref{lGCD}.
\end{proof}

With Lemma \ref{lG} and Lemma \ref{lGCD} in hand we return to estimating the number of $\beta \in \MM$ satisfying $|\beta^{(i)}| \leq x^{\frac{1}{n}}$ and (\ref{eException}). We choose a $\sigma \in S$ and we will consider it as fixed for the remainder of the proof. Note that any integer $n > 0$ can be factored uniquely as
\begin{align*}
n = q' g' r',
\end{align*}
where $q'$ is a squarefree integer coprime to $mF$, $g'$ is a squarefull integer coprime to $mF$ and $r'$ is composed entirely of primes from $mF$. This allows us to define $\text{sqf}(n, mF) := q'$. We start by giving an upper bound for
\begin{align*}
\left|\left\{\beta \in \MM: |\beta^{(i)}| \leq x^{\frac{1}{n}}, \text{sqf}(\Norm{(\beta - \sigma(\beta))}, mF) \leq Z\right\}\right|.
\end{align*}
To do this, we need a slight generalization of the argument on \cite[p.\ 729]{FIMR}. Recall that $K^{\sigma}$ is the subfield of $K$ fixed by $\sigma$ and $\OO_{K^\sigma}$ its ring of integers. Decompose $\OO_K$ as
\begin{align*}
\OO_K = \OO_{K^\sigma} \oplus \MM'.
\end{align*}
Then we have
\begin{multline}
\label{eS1}
\left|\left\{\beta \in \MM: |\beta^{(i)}| \leq x^{\frac{1}{n}}, \text{sqf}(\Norm{(\beta - \sigma(\beta))}, mF) \leq Z\right\}\right| \\
\ll x^{\frac{1}{\text{ord}(\sigma)} - \frac{1}{n}} \left|\left\{\beta \in \MM': |\beta^{(i)}| \leq x^{\frac{1}{n}}, \text{sqf}(\Norm{(\beta - \sigma(\beta))}, mF) \leq Z\right\}\right|.
\end{multline}
The map $\MM' \rightarrow \OO_K$ given by $\beta \mapsto \beta - \sigma(\beta)$ is injective. Set $\gamma := \beta - \sigma(\beta)$. Furthermore, the conjugates of $\gamma$ satisfy $|\gamma^{(i)}| \leq 2x^{\frac{1}{n}}$, which gives
\begin{multline}
\label{eS2}
\left|\left\{\beta \in \MM': |\beta^{(i)}| \leq x^{\frac{1}{n}}, \text{sqf}(\Norm{(\beta - \sigma(\beta))}, mF) \leq Z\right\}\right| \\
\leq \left|\left\{\gamma \in \OO_K: |\gamma^{(i)}| \leq 2x^{\frac{1}{n}}, \text{sqf}(\Norm{(\gamma)}, mF) \leq Z\right\}\right|.
\end{multline}
Instead of counting algebraic integers $\gamma$, we will count the principal ideals they generate, where each given ideal occurs no more than $\ll (\log x)^n$ times. This yields the bound
\begin{multline}
\left|\left\{\gamma \in \OO_K: |\gamma^{(i)}| \leq 2x^{\frac{1}{n}}, \text{sqf}(\Norm{(\gamma)}, mF) \leq Z\right\}\right| \\
\ll (\log x)^n \left|\left\{\mathfrak{b} \subseteq \OO_K: \Norm{(\mathfrak{b})} \leq 2^n x, \text{sqf}(\Norm{(\mathfrak{b})}, mF) \leq Z\right\}\right| \nonumber.
\end{multline}
We conclude that
\begin{align}
\label{eS4}
\left|\left\{\gamma \in \OO_K: |\gamma^{(i)}| \leq 2x^{\frac{1}{n}}, \text{sqf}(\Norm{(\gamma)}, mF)\leq Z\right\}\right| \ll (\log x)^n \sum_{\substack{b \leq 2^nx \\ \text{sqf}(b, mF) \leq Z}} \tau_K(b),
\end{align}
where we remind the reader that $\tau_K(b)$ denotes the number of ideals in $K$ of norm $b$. 

Let us count the number of $b \leq 2^n x$ satisfying $\text{sqf}(b, mF) \leq Z$. We do this by counting the number of possible $g', r' \leq 2^n x$ that can occur in the factorization $b = q'g'r'$. First of all, there are $\ll x^{\frac{1}{2}}$ squarefull integers $g'$ satisfying $g' \leq 2^n x$. To bound the number of $r' \leq 2^n x$, we observe that we may assume $m \leq x$, because otherwise the sum in (\ref{eFirstSum}) is empty. This implies that the number of integers $r' \leq 2^n x$ that are composed entirely of primes from $mF$ is $\ll_\epsilon x^\epsilon$. Obviously there are at most $Z$ squarefree integers $q'$ coprime to $mF$ satisfying $q' \leq Z$. We conclude that the number of $b \leq 2^n x$ satisfying $\text{sqf}(b, mF) \leq Z$ is $\ll_\epsilon Zx^{\frac{1}{2} + \epsilon}$. Combined with the upper bound $\tau_K(b) \ll_\epsilon x^\epsilon$ we obtain
\begin{align}
\label{eS5}
(\log x)^n \sum_{\substack{b \leq 2^nx \\ \text{sqf}(b, mF) \leq Z}} \tau_K(b) \ll_\epsilon Zx^{\frac{1}{2} + \epsilon}.
\end{align}
Stringing together the inequalities (\ref{eS1}), (\ref{eS2}), (\ref{eS4}) and (\ref{eS5}) we conclude that
\begin{align}
\label{eS6}
\left|\left\{\beta \in \MM: |\beta^{(i)}| \leq x^{\frac{1}{n}}, \text{sqf}(\Norm{(\beta - \sigma(\beta))}, mF) \leq Z\right\}\right| \ll_\epsilon Z x^{\frac{1}{2} + \frac{1}{\text{ord}(\sigma)} - \frac{1}{n} + \epsilon}.
\end{align}
Now in order to give an upper bound for the number of $\beta$ satisfying $|\beta^{(i)}| \leq x^{\frac{1}{n}}$ and (\ref{eException}), that is
\begin{align*}
p \mid \prod_{\sigma \in S} \Norm{(\beta - \sigma(\beta))} \Rightarrow p^2 \mid mF\prod_{\sigma \in S} \Norm{(\beta - \sigma(\beta))},
\end{align*}
we start by picking $Z = x^{\frac{1}{3n}}$ and discarding all $\beta$ satisfying (\ref{eS6}) for the $\sigma \in S$ we fixed earlier. For this $\sigma \in S$ and varying $\tau \in S$ with $\tau \neq \sigma$ we apply Lemma \ref{lGCD} to obtain
\begin{align}
\label{eDiscard}
|\{\beta \in \MM: |\beta^{(i)}| \leq x^{\frac{1}{n}}, \gcd(\Norm{(\beta - \sigma(\beta))}, \Norm{(\beta - \tau(\beta))}) > x^{\frac{1}{3n|S|}}\}| \ll_\epsilon x^{\frac{n - 1}{n} -\frac{1}{54n|S|} + \epsilon}.
\end{align}
We further discard all $\beta$ satisfying (\ref{eDiscard}) for some $\tau \in S$ with $\tau \neq \sigma$. Now it is easily checked that the remaining $\beta$ do not satisfy (\ref{eException}). Hence we have completed our task of estimating the number of $\beta$ satisfying $|\beta^{(i)}| \leq x^{\frac{1}{n}}$ and (\ref{eException}).

Let $A_0(x; \rho)$ be the contribution to $A(x; \rho)$ of the terms $\alpha = a + \beta$ for which (\ref{eException}) does not hold and let $A_{\square}(x; \rho)$ be the contribution to $A(x; \rho)$ for which (\ref{eException}) holds. Then we have the obvious identity
\begin{align*}
A(x; \rho) = A_0(x; \rho) + A_{\square}(x; \rho).
\end{align*}
Next we make a further partition
\begin{align*}
A_0(x; \rho) = A_1(x; \rho) + A_2(x; \rho),
\end{align*}
where the components run over $\alpha = a + \beta$, $\beta \in \MM$ with $\beta$ such that
\begin{align*}
g_0 &\leq Y \text{ in } A_1(x; \rho) \\
g_0 &> Y \text{ in } A_2(x; \rho).
\end{align*}
Here $Y$ is at our disposal and we choose it later. From (\ref{eS6}) and (\ref{eDiscard}) we deduce that
\begin{align*}
A_{\square}(x; \rho) \ll_\epsilon x^{1 -\frac{1}{54n|S|} + \epsilon}.
\end{align*}
To estimate $A_1(x; \rho)$ we apply \ref{eFirstBound} and sum over all $\beta \in \MM$ satisfying $|\beta^{(i)}| \leq x^{\frac{1}{n}}$, ignoring all other restrictions on $\beta$, to obtain
\begin{align*}
A_1(x; \rho) \ll_\epsilon Y x^{1 - \frac{\delta}{n} + \epsilon}.
\end{align*}
We still have to bound $A_2(x; \rho)$. Recall that
\begin{align*}
\ccc = \prod_{\sigma \in S} \ccc(\sigma, \beta),
\end{align*}
leading to the factorization $\ccc = \mathfrak{g} \qq$ in (\ref{eFactorc}). We further recall that $g_0$ is the radical of $\Norm{\mathfrak{g}}$. Now factor each term $\ccc(\sigma, \beta)$ as
\begin{align}
\label{eFactorcsb}
\ccc(\sigma, \beta) = \mathfrak{g}(\sigma, \beta) \qq(\sigma, \beta)
\end{align}
just as in (\ref{eFactorc}). The point of (\ref{eFactorcsb}) is that
\begin{align*}
\mathfrak{g} \mid \prod_{\sigma \in S} \mathfrak{g}(\sigma, \beta) \prod_{\substack{\sigma, \tau \in S \\ \sigma \neq \tau}} \gcd(\mathfrak{c}(\sigma, \beta), \mathfrak{c}(\tau, \beta))
\end{align*}
and therefore
\begin{align*}
g_0 \mid \prod_{\sigma \in S} g_0(\sigma, \beta) \prod_{\substack{\sigma, \tau \in S \\ \sigma \neq \tau}} \gcd(\mathfrak{c}(\sigma, \beta), \mathfrak{c}(\tau, \beta)).
\end{align*}
We use Lemma \ref{lG} to discard all $\beta$ satisfying $g_0(\sigma, \beta) > Y^{\frac{1}{|S|^2}}$. Similarly, we use Lemma \ref{lGCD} to discard all $\beta$ satisfying $\gcd(\mathfrak{c}(\sigma, \beta), \mathfrak{c}(\tau, \beta)) > Y^{\frac{1}{|S|^2}}$. Then the remaining $\beta$ satisfy $g_0 \leq Y$. Furthermore, we have removed
\begin{align*}
\ll_\epsilon x^{\frac{n - 1}{n} + \epsilon} Y^{-\frac{1}{18|S|^2}} + x^{\frac{n - 2}{n}} + Y^{\frac{2n - 4}{3|S|^2}} +  x^{\frac{n - 1}{n} + \epsilon} Y^{-\frac{1}{3|S|^2}}
\end{align*}
$\beta$ in total and hence
\begin{align*}
A_2(x; \rho) \ll_\epsilon x^{1 + \epsilon} Y^{-\frac{1}{18|S|^2}} + x^{\frac{n - 1}{n}} + x^{\frac{1}{n}}Y^{\frac{2n - 4}{3|S|^2}} +  x^{1 + \epsilon} Y^{-\frac{1}{3|S|^2}}.
\end{align*}
After picking $Y = x^{\frac{\delta}{2n}}$ we conclude that
\begin{align*}
A(x) \ll_\epsilon x^{1 - \frac{\delta}{54n|S|^2} + \epsilon}.
\end{align*}
We will now sketch how to modify this proof for totally complex $K$. We have to bound
\begin{align}
\label{eFirstSum2}
A(x) = \sum_{\substack{\Norm{\mathfrak{a}} \leq x \\ (\mathfrak{a}, F) = 1, \mathfrak{m} \mid \mathfrak{a}}} r(\aaa)\sum_{t\in T_K}\sum_{v\in V_K/V_K^2}\psi(tv\alpha\bmod F)\prod_{\sigma\in S}\spin(\sigma, tv\alpha).
\end{align}
We use the fundamental domain constructed for totally complex fields form subsection \ref{fundDom} and we pick for each principal $\mathfrak{a}$ its generator in $\mathcal{D}$. Then equation (\ref{eFirstSum2}) becomes
\begin{align*}
A(x) &= \sum_{t\in T_K}\sum_{v\in V_K/V_K^2} \sum_{\substack{\alpha \in \DD, \Norm{\alpha} \leq x \\ \alpha \equiv \rho \bmod F \\ \alpha \equiv 0 \bmod \mathfrak{m}}} \psi(tv\alpha \bmod F) \prod_{\sigma\in S}\spin(\sigma, tv\alpha) \\
&= \sum_{t\in T_K}\sum_{v\in V_K/V_K^2} \sum_{\substack{\alpha \in tv\DD, \Norm{\alpha} \leq x \\ \alpha \equiv \rho \bmod F \\ \alpha \equiv 0 \bmod \mathfrak{m}}} \psi(\alpha \bmod F) \prod_{\sigma\in S}\spin(\sigma, \alpha).
\end{align*}
We deal with each sum of the shape
\begin{align}
\label{eReduceReal}
\sum_{\substack{\alpha \in tv\DD, \Norm{\alpha} \leq x \\ \alpha \equiv \rho \bmod F \\ \alpha \equiv 0 \bmod \mathfrak{m}}} \psi(\alpha \bmod F) \prod_{\sigma\in S}\spin(\sigma, \alpha)
\end{align}
exactly in the same way as for real quadratic fields $K$, where it is important to note that the shifted fundamental domain $tv\DD$ still has the essential properties we need. Combining our estimate for each sum in equation (\ref{eReduceReal}), we obtain the desired upper bound for $A(x)$.

\section{Bilinear sums}
\label{sBilinear}
Let $x, y>0$ and let $\{v_{\aaa}\}_{\aaa}$ and $\{w_{\bb}\}_{\bb}$ be two sequences of complex numbers bounded in modulus by $1$. Define
\begin{equation}\label{BS1}
B(x, y) = \sum_{\Norm(\aaa)\leq x}\sum_{\Norm(\bb)\leq y}v_{\aaa}w_{\bb}s_{\aaa\bb}.
\end{equation}
We wish to prove that for all $\epsilon>0$, we have 
\begin{equation}\label{BS2}
B(x, y) \ll_{\epsilon} \left(x^{-\frac{1}{6n}}+y^{-\frac{1}{6n}}\right)\left(xy\right)^{1+\epsilon},
\end{equation}
where the implied constant is uniform in all choices of sequences $\{v_{\aaa}\}_{\aaa}$ and $\{w_{\bb}\}_{\bb}$ as above. 

We split the sum $B(x, y)$ into $h^2$ sums according to which ideal classes $\aaa$ and $\bb$ belong to. In fact, since $s_{\aaa\bb}$ vanishes whenever $\aaa\bb$ does not belong to the principal class, it suffices to split $B(x, y)$ into $h$ sums
$$
B(x, y) = \sum_{i = 1}^{h}B_i(x, y),\quad B_i(x, y) = \sum_{\substack{\Norm(\aaa)\leq x \\ \aaa\in C_i}}\sum_{\substack{\Norm(\bb)\leq y \\ \bb\in C_i^{-1}}}v_{\aaa}w_{\bb}s_{\aaa\bb}.
$$
We will prove the desired estimate for each of the sums $B_i(x, y)$. So fix an index $i\in\{1, \ldots, h\}$, let $\AAA \in \Cl_a$ be the ideal belonging to the ideal class $C_i^{-1}$, and let $\BBB\in \Cl_b$ be the ideal belonging to the ideal class $C_i$. The conditions on $\aaa$ and $\bb$ above mean that
$$
\aaa\AAA = (\alpha),\quad \alpha\succ 0
$$
and
$$
\bb\BBB = (\beta),\quad \beta\succ 0.
$$
Since $\AAA\in C_i^{-1}$ and $\BBB\in C_i$, there exists an element $\gamma\in\OO_K$ such that
$$
\AAA\BBB = (\gamma),\quad\gamma\succ 0.
$$
We are now in a position to use the factorization formula for $\spin(\aaa\bb)$ appearing in \cite[(3.8), p. 708]{FIMR}, which in turn leads to a factorization formula for $s_{\aaa\bb}$. We note that the formula \cite[(3.8), p. 708]{FIMR} also holds in case $K$ is totally complex, with exactly the same proof. We have
\begin{equation}\label{spinfactor}
\spin(\sigma, \alpha\beta/\gamma) = \spin(\sigma, \gamma)\delta(\sigma; \alpha, \beta)\left(\frac{\alpha\gamma}{\sigma(\aaa\BBB)}\right)\left(\frac{\beta\gamma}{\sigma(\bb\AAA)}\right)\left(\frac{\alpha}{\sigma(\beta)\sigma^{-1}(\beta)}\right),
\end{equation}
where $\delta(\sigma; \alpha, \beta)\in \{\pm 1\}$ is a factor which comes from an application of quadratic reciprocity and which depends only on $\sigma$ and the congruence classes of $\alpha$ and $\beta$ modulo $8$.

If $K$ is real quadratic, then we set
$$
v_{\aaa}' = v_{\aaa}\prod_{\sigma\in S} \left(\frac{\alpha\gamma}{\sigma(\aaa\BBB)}\right), \quad w_{\bb}' = w_{\bb}\prod_{\sigma\in S} \left(\frac{\beta\gamma}{\sigma(\bb\AAA)}\right),
$$
and
$$
\delta(\alpha, \beta) = \psi(\alpha\beta\bmod F)\prod_{\sigma\in S}\delta(\sigma; \alpha, \beta),\quad s(\gamma) = \prod_{\sigma\in S}\spin(\sigma, \gamma),
$$
so that we can rewrite the sum $B_i(x, y)$ as
\begin{equation}\label{BSreal}
B_i(x, y) = s(\gamma) \sum_{\substack{\alpha\in \DD \\ \Norm(\alpha)\leq x\Norm(\AAA) \\ \alpha\equiv 0\bmod \AAA}}\sum_{\substack{\beta\in\DD \\ \Norm(\beta)\leq y\Norm(\BBB) \\ \beta\equiv 0\bmod \BBB}}\delta(\alpha, \beta)v_{(\alpha)/\AAA}'w_{(\beta)/\BBB}'\prod_{\sigma\in S}\left(\frac{\alpha}{\sigma(\beta)\sigma^{-1}(\beta)}\right).
\end{equation}
Now set
$$
v_{\alpha} = \ONE(\alpha\equiv 0\bmod \AAA) \cdot v_{(\alpha)/\AAA}'
$$
and
$$
w_{\beta} = \ONE(\beta\equiv 0\bmod \BBB)\cdot w_{(\beta)/\BBB}', 
$$
where $\ONE(P)$ is the indicator function of a property $P$. Also, for $\alpha, \beta\in\OO_K$ with $\beta$ odd, we define
$$
\phi(\alpha, \beta) = \prod_{\sigma\in S}\left(\frac{\alpha}{\sigma(\beta)\sigma^{-1}(\beta)}\right).
$$
Finally, we further split $B_i(x, y)$ according to the congruence classes of $\alpha$ and $\beta$ modulo $F$, so as to control the factor $\delta(\alpha, \beta)$, which now depends on congruence classes of $\alpha$ and $\beta$ modulo $F$ due to the presence of $\psi(\alpha\beta\bmod F)$. We have 
$$
B_i(x, y) = s(\gamma) \sum_{\alpha_0\in (\OO_K/(F))^{\times}}\sum_{\beta_0\in (\OO_K/(F))^{\times}}\delta(\alpha_0, \beta_0)B_i(x, y; \alpha_0, \beta_0), 
$$
where 
$$
B_i(x, y; \alpha_0, \beta_0) = \sum_{\substack{\alpha\in \DD(x\Norm(\AAA)) \\ \alpha\equiv \alpha_0\bmod F}}\sum_{\substack{\beta\in\DD(y\Norm(\BBB)) \\ \beta\equiv \beta_0\bmod F}}v_{\alpha}w_{\beta}\phi(\alpha, \beta).
$$
To prove the bound \eqref{BS2}, at least in the case that $K$ is totally real, it now suffices to prove, for each $\epsilon>0$, the bound
\begin{equation}\label{BS3}
B_i(x, y; \alpha_0, \beta_0) \ll_{\epsilon} \left(x^{-\frac{1}{6n}}+y^{-\frac{1}{6n}}\right)\left(xy\right)^{1+\epsilon},
\end{equation}
where the implied constant is uniform in all choices of uniformly bounded sequences of complex numbers $\{v_{\alpha}\}_{\alpha}$ and $\{w_{\beta}\}_{\beta}$ indexed by elements of $\OO_K$. Each of the sums $B_i(x, y; \alpha_0, \beta_0)$ is of the same shape as $B(M, N; \omega, \zeta)$ in \cite[(3.2)]{KM2}; in the notation of \cite[Section 3]{KM2}, $\ff = (F)$, $\alpha_w$ corresponds to $v_{\alpha}$, $\beta_z$ corresponds to $w_{\beta}$, and $\gamma(w, z)$ corresponds to $\phi(\alpha, \beta)$ (unfortunately with the arguments $\alpha$ and $\beta$ flipped). Our desired estimate for $B_i(x, y; \alpha_0, \beta_0)$, and hence also $B(x, y)$, would now follow from \cite[Proposition 3.6]{KM2}, provided that we can verify properties (P1)-(P3) for the function $\phi(\alpha, \beta)$. We will verify a slightly weaker modified version of property (P3), namely that 
$$
\sum_{\xi\bmod \Norm(\beta)}\phi(\xi, \beta) = 0
$$
whenever $|\Norm(\beta)|$ is not \textit{squarefull} in $\Z$. Recall that an integer $n$ is called squarefull if whenever a prime $p$ divides $n$, then $p^2$ divides $n$. Since an integer that is not squarefull cannot be a square, this condition is weaker than (P3) in \cite{KM2}. Nonetheless, \cite[Proposition 3.6]{KM2} still holds true even with this weaker assumption. For completeness, we state it here.
\begin{lemma}\label{Prop36}
Let $\DD_1$ and $\DD_2$ be a pair of translates of $\DD$, i.e., $\DD_i = v_i\DD$ for some $v_i\in V_K$. Let $\ff$ be a non-zero ideal in $\OO_K$, and let $S_{\ff}$ be the set of ideals in $\OO_K$ coprime to $\ff$. Suppose $\gamma$ is a map 
$$
\gamma: S_{\ff}\times\OO_F\rightarrow \{-1, 0, 1\}
$$
satisfying the following properties:
\\\\
(P1) for every pair of invertible congruence classes $\omega$ and $\zeta$ modulo $\ff$, there exists $\mu(\omega, \zeta)\in\{\pm 1\}$ such that $\gamma(w, z) = \mu(\omega, \zeta)\gamma(z, w)$ whenever $w \equiv \omega\bmod \ff$ and $z \equiv \zeta\bmod \ff$;
\\\\
(P2) for all $z_1, z_2\in\OO_K$ and all $w\in S_{\ff}$, we have $\gamma(w, z_1z_2) = \gamma(w, z_1)\gamma(w, z_2)$; similarly, for all $w_1, w_2\in S_{\ff}$ and all $z\in\OO_K$, we have $\gamma(w_1w_2, z) = \gamma(w_1, z)\gamma(w_2, z)$; and
\\\\
(P3) for all non-zero $w\in\OO_K$, if $z_1, z_2\in\OO_K$ satisfy $z_1\equiv z_2\bmod \Norm(w)$, then we have $\gamma(w, z_1) = \gamma(w, z_2)$; moreover, if $|\Norm(w)|$ is not squarefull in $\Z$, then $\sum_{\xi\bmod w}\gamma(w, \xi) = 0$.
Let
$$
\mathcal{B}(M, N; \omega, \zeta) = \sum_{\substack{w\in\DD_1,\ \Norm(w)\leq M \\ w\equiv \omega\bmod \ff}}\sum_{\substack{z\in\DD_2,\ \Norm(z)\leq N \\ z\equiv \zeta\bmod \ff}}\alpha_w\beta_z\gamma(w, z),
$$
where $\{\alpha_w\}_w$ and $\{\beta_z\}_z$ are bounded sequences of complex numbers, $\omega$ and $\zeta$ are invertible congruence classes modulo $\ff$, and $M$ and $N$ are positive real numbers. Then
$$ 
\mathcal{B}(M, N; \omega, \zeta) \ll_{\epsilon}\left(M^{-\frac{1}{6n}}+N^{-\frac{1}{6n}}\right)(MN)^{1+\epsilon},
$$
where the implied constant depends on $\epsilon$, on the units $v_1$ and $v_2$, on the supremum norms of $\{\alpha_w\}_w$ and $\{\beta_z\}_z$, and on the congruence classes $\omega$ and $\zeta$ modulo $\ff$.
\end{lemma}
\begin{proof}
See the proof of \cite[Proposition 3.6]{KM2}. Indeed, as can be seen on top of page 13 of \cite{KM2}, the only feature of squares that we used is that there are $\ll M$ squares of size at most $M^2$; the same estimate is true for squarefull numbers. 
\end{proof}
We now verify (P1)-(P3), thereby proving the bound \eqref{BS3} and hence also the bound \eqref{BS2}. Property (P1) follows from the law of quadratic reciprocity, since for odd $\alpha$ and $\beta$ we have 
\begin{align*}
\phi(\alpha, \beta) & = \prod_{\sigma\in S}\left(\frac{\alpha}{\sigma(\beta)}\right)\left(\frac{\alpha}{\sigma^{-1}(\beta)}\right) \\
& = \prod_{\sigma\in S}\mu(\sigma; \alpha, \beta)\left(\frac{\sigma(\beta)}{\alpha}\right)\left(\frac{\sigma^{-1}(\beta)}{\alpha}\right) \\
& = \left(\prod_{\sigma\in S}\mu(\sigma; \alpha, \beta)\right)\cdot \prod_{\sigma\in S}\left(\frac{\beta}{\sigma^{-1}(\alpha)}\right)\left(\frac{\beta}{\sigma(\alpha)}\right) \\
& = \left(\prod_{\sigma\in S}\mu(\sigma; \alpha, \beta)\right)\cdot \phi(\beta, \alpha),
\end{align*}
where $\mu(\sigma; \alpha, \beta)$ depends only on $\sigma$ and the congruence classes of $\alpha$ and $\beta$ modulo $8$. Property~(P2) follows immediately from the multiplicativity of each argument of the quadratic residue symbol $(\cdot/ \cdot)$. Finally, for property (P3), since $\sigma^{-1}\not\in S$ whenever $\sigma\in S$, we see that 
$$
\varphi(\beta) = \prod_{\sigma\in S}\sigma(\beta)\sigma^{-1}(\beta)
$$
divides $\Norm(\beta) = \prod_{\sigma\in\Gal(K/\Q)}\sigma(\beta)$; thus, the first part of (P3) indeed holds true. It now suffices to prove that
$$
\sum_{\xi\bmod \Norm(\beta)}\left(\frac{\xi}{\varphi(\beta)}\right)
$$
vanishes if $|\Norm(\beta)|$ is not squarefull.  The sum above is a multiple of the sum
$$
\sum_{\xi\bmod \varphi(\beta)}\left(\frac{\xi}{\varphi(\beta)}\right),
$$
which vanishes if the principal ideal generated by $\varphi(\beta)$ is not the square of an ideal. The proof now proceeds as in \cite[Lemma 3.1]{FIMR}. Supposing $|\Norm(\beta)|$ is not squarefull, we take a rational prime $p$ such that $p \mid \Norm(\beta)$ but $p^2\nmid \Norm(\beta)$. This implies that there is a degree-one prime ideal divisor $\pp$ of $\beta$ such that $(\beta) = \pp\ccc$ with $\ccc$ coprime to $p$, i.e., coprime to all the conjugates of $\pp$. Hence $\varphi(\beta)$ factors as
$$
(\varphi(\beta)) = \prod_{\sigma\in S}\sigma(\pp)\sigma^{-1}(\pp) \prod_{\sigma\in S}\sigma(\ccc)\sigma^{-1}(\ccc),
$$ 
where the evidently non-square $\prod_{\sigma\in S}\sigma(\pp)\sigma^{-1}(\pp)$ is coprime to $\prod_{\sigma\in S}\sigma(\ccc)\sigma^{-1}(\ccc)$, hence proving that $(\varphi(\beta))$ is not a square. This proves that property (P3) holds true, and then Lemma~\ref{Prop36} implies the estimate \eqref{BS3} and hence also \eqref{BS2}, at least in the case that $K$ is totally real.

If $K$ is totally complex, fix $t\in T_K$ and $v\in V_K/V_K^2$. Then replacing $\alpha$ by $tv\alpha$ in \eqref{spinfactor}, we get
$$
\spin(\sigma, tv\alpha\beta/\gamma) = \spin(\sigma, \gamma)\delta(\sigma; tv\alpha, \beta)\left(\frac{tv\alpha\gamma}{\sigma(\aaa\BBB)}\right)\left(\frac{\beta\gamma}{\sigma(\bb\AAA)}\right)\left(\frac{tv}{\sigma(\beta)\sigma^{-1}(\beta)}\right)\left(\frac{\alpha}{\sigma(\beta)\sigma^{-1}(\beta)}\right),
$$
where now $\delta(\sigma; \alpha, \beta; t, v) = \delta(\sigma; tv\alpha, \beta)\left(\frac{tv}{\sigma(\beta)\sigma^{-1}(\beta)}\right)\in \{\pm 1\}$ depends only on $\sigma$, $t$, $v$, and the congruence classes of $\alpha$ and $\beta$ modulo $8$. Then instead of \eqref{BSreal}, we have $B_i(x, y) = $
\begin{equation}\label{BScomplex}
s(\gamma)\sum_{t\in T_K}\sum_{v\in V_K/V_K^2} \sum_{\substack{\alpha\in \DD \\ \Norm(\alpha)\leq x\Norm(\AAA) \\ \alpha\equiv 0\bmod \AAA}}\sum_{\substack{\beta\in\DD \\ \Norm(\beta)\leq y\Norm(\BBB) \\ \beta\equiv 0\bmod \BBB}}\delta(\alpha, \beta; t, v)v(t, v)_{(\alpha)/\AAA}'w_{(\beta)/\BBB}'\prod_{\sigma\in S}\left(\frac{\alpha}{\sigma(\beta)\sigma^{-1}(\beta)}\right),
\end{equation}
where now
$$
v(t, v)_{\aaa}' = v_{\aaa}\prod_{\sigma\in S} \left(\frac{tv\alpha\gamma}{\sigma(\aaa\BBB)}\right), \quad w_{\bb}' = w_{\bb}\prod_{\sigma\in S} \left(\frac{\beta\gamma}{\sigma(\bb\AAA)}\right),
$$
and
$$
\delta(\alpha, \beta; t, v) = \psi(tv\alpha\beta\bmod F)\prod_{\sigma\in S}\delta(\sigma; \alpha, \beta; t, v),\quad s(\gamma) = \prod_{\sigma\in S}\spin(\sigma, \gamma).
$$
The rest of the proof now proceeds identically to the case when $K$ is totally real.

\section{Governing fields}
\label{noGovField}
Let $E = \Q(\zeta_8, \sqrt{1 + i})$ and let $h(-4p)$ be the class number of $\Q(\sqrt{-4p})$. It is well-known that $E$ is a governing field for the $8$-rank of $\Q(\sqrt{-4p})$; in fact $8$ divides $h(-4p)$ if and only if $p$ splits completely in $E$. We assume that $K$ is a hypothetical governing field for the $16$-rank of $\Q(\sqrt{-4p})$ and derive a contradiction. If $K'$ is a normal field extension of $\Q$ containing $K$, then $K'$ is also a governing field. Therefore we can reduce to the case that $K$ contains $E$. In particular, $K$ is totally complex.

We have $\Gal(E/\Q) \cong D_4$ and we fix an element of order $4$ in $\Gal(E/\Q)$ that we call $r$. Let $p$ be a rational prime that splits completely in $E$. Since $E$ is a PID, we can take $\pi$ to be a prime in $\OO_E$ above $p$. It follows from Proposition 6.2 of \cite{KM1}, which is based on earlier work of Bruin and Hemenway \cite{BH}, that there exists an integer $F$ and a function $\psi_0: (\OO_E/F\OO_E)^{\times}\rightarrow \mathbb{C}$ such that for all $p$ with $(p, F) = 1$ we have
\begin{align}
\label{e16c}
16 \mid h(-4p) \Leftrightarrow \psi_0(\pi \bmod F) \left(\frac{r(\pi)}{\pi}\right)_{E, 2} = 1,
\end{align}
where $\psi_0(\alpha \bmod F) = \psi_0(\alpha u^2 \bmod F)$ for all $\alpha\in\OO_K$ coprime to $F$ and all $u\in\OO_K^{\times}$. We take $S$ equal to the inverse image of our fixed automorphism $r$ under the natural surjective map $\Gal(K/\Q) \rightarrow \Gal(E/\Q)$. Then it is easily seen that $\sigma \in S$ implies $\sigma^{-1} \not \in S$. If $\pp$ is a principal prime of $K$ with generator $w$ of norm $p$, we have
\begin{align*}
\prod_{\sigma \in S} \spin(\sigma, w) &= \prod_{\sigma \in S} \left(\frac{w}{\sigma(w)}\right)_{K, 2} = \left(\frac{w}{r(\Norm_{K/E}(w))}\right)_{K, 2} = \psi_1(w \bmod 8) \left(\frac{r(\Norm_{K/E}(w))}{w} \right)_{K, 2} \\
&= \psi_1(w \bmod 8) \left(\frac{r(\Norm_{K/E}(w))}{\Norm_{K/E}(w)}\right)_{E, 2}.
\end{align*}
We are now going to apply Theorem \ref{tJoint} to the number field $K$, the function
\[
\psi(w \bmod F) := \psi_1(w \bmod 8) \psi_0\left(\Norm_{K/E}(w) \bmod F\right).
\]
and $S$ as defined above. Then for a principal prime $\pp$ of $K$ with generator $w$ and norm $p$
\begin{align}
\label{e16above}
s_\pp = \sum_{t \in T_K}\sum_{v \in V_K/V_K^2} \psi\left(tvw \bmod F\right) \prod_{\sigma \in S}\spin(\sigma, tvw) = 2|T_K| |V_K/V_K^2| \left(\mathbf{1}_{16 \mid h(-p)} - \frac{1}{2}\right),
\end{align}
since the equivalence in (\ref{e16c}) does not depend on the choice of $\pi$. Theorem \ref{tJoint} shows oscillation of the sum
\[
\sum_{\substack{\Norm(\pp)\leq X \\ \pp\text{ principal}}} s_\pp.
\]
The dominant contribution of this sum comes from prime ideals of degree $1$ and for these primes equation (\ref{e16above}) is valid. But if $K$ were to be a governing field, $s_\pp$ has to be constant on unramified prime ideals of degree $1$, which is the desired contradiction.

\end{document}